\documentclass[reqno, 12pt,a4letter]{amsart}
\usepackage{amsmath,amsxtra,amssymb,latexsym, amscd,amsthm}
\usepackage{graphicx,color}
\usepackage[utf8]{inputenc}
\usepackage[mathscr]{euscript}
\usepackage{mathrsfs}
\usepackage{tikz-cd}
\usepackage[english]{babel}
\usepackage{enumerate}
\newtheorem {theorem}{Theorem}[section]
\newtheorem {corollary}{Corollary}[section]
\newtheorem {lemma}{Lemma}[section]
\newtheorem {example}{Example}[section]
\newtheorem {definition}{Definition}[section]
\newtheorem {remark}{Remark}[section]

\def\ar{a\kern-.370em\raise.16ex\hbox{\char95\kern-0.53ex\char'47}\kern.05em}
\def\ees{{\accent"5E e}\kern-.385em\raise.2ex\hbox{\char'23}\kern-.08em}
\def\eex{{\accent"5E e}\kern-.470em\raise.3ex\hbox{\char'176}}
\def\AR{A\kern-.46em\raise.80ex\hbox{\char95\kern-0.53ex\char'47}\kern.13em}
\def\EES{{\accent"5E E}\kern-.5em\raise.8ex\hbox{\char'23 }}
\def\EEX{{\accent"5E E}\kern-.60em\raise.9ex\hbox{\char'176}\kern.1em}
\def\ow{o\kern-.42em\raise.82ex\hbox{
  \vrule width .12em height .0ex depth .075ex \kern-0.16em \char'56}\kern-.07em}
\def\OW{O\kern-.460em\raise1.36ex\hbox{
\vrule width .13em height .0ex depth .075ex \kern-0.16em \char'56}\kern-.07em}
\def\UW{U\kern-.42em\raise1.36ex\hbox{
\vrule width .13em height .0ex depth .075ex \kern-0.16em \char'56}\kern-.07em}
\def\DD{D\kern-.7em\raise0.4ex\hbox{\char '55}\kern.33em}
\title{Local minimizers of semi-algebraic functions from the viewpoint of tangencies}

\author{TI\EES N-S\OW N PH\d{A}M}
\address{Department of Mathematics, University of Dalat, 1 Phu Dong Thien Vuong, Dalat, Vietnam}
\email{sonpt@dlu.edu.vn}

\date{\today}

\thanks{The authors are partially supported by Vietnam National Foundation for Science and Technology Development (NAFOSTED), grant 101.04-2019.302}
 
\begin{document}

\begin{abstract}
Consider a semi-algebraic function $f\colon\mathbb{R}^n \to {\mathbb{R}},$ which is continuous around a point $\bar{x} \in \mathbb{R}^n.$ Using the so--called {\em tangency variety} of $f$ at $\bar{x},$ we first provide necessary and sufficient conditions for $\bar{x}$ to be a local minimizer of $f,$ and then in the case where $\bar{x}$ is an isolated local minimizer of $f,$ we define a ``tangency exponent'' $\alpha_* > 0$ so that for any $\alpha \in \mathbb{R}$
the following four conditions are always equivalent:
\begin{enumerate}
\item[{(i)}] the inequality $\alpha \ge \alpha_*$ holds;
\item[{(ii)}] the point $\bar{x}$ is an $\alpha$th order sharp local minimizer of $f;$ 
\item[{(iii)}] the limiting subdifferential $\partial f$ of $f$ is $(\alpha - 1)$th order strongly metrically subregular at $\bar{x}$ for $0;$ and
\item[{(iv)}] the function $f$ satisfies the \L ojaseiwcz gradient inequality at $\bar{x}$ with the exponent $1 - \frac{1}{\alpha}.$
\end{enumerate}
Besides, we also present a counterexample to a conjecture posed by Drusvyatskiy and Ioffe [Math. Program. Ser. A, 153(2):635--653, 2015].
\end{abstract}
\maketitle

\section{Introduction}

Optimality conditions form the foundations of mathematical programming both theoretically and computationally (see, for example, \cite{Bertsekas1999, Bonnas2000, Clarke1990, Lasserre2015, Mordukhovich2006, Ruszczynski2006}).

To motivate the discussion, consider a function $f \colon \mathbb{R}^n \to {\mathbb{R}},$ which is continuous around a point $\bar{x} \in \mathbb{R}^n.$ It is well known that if $\bar{x}$ is a local minimizer of $f,$ then $0$ belongs to the limiting subdifferential $\partial f(\bar{x})$ of $f$ at $\bar{x}$ (see the next section for definitions and notations). The converse is known to be true for convex functions, but it is false in the general case.

On the other hand, when $f$ is a polynomial function, Barone-Netto defined in \cite{BaroneNetto1984} a finite family of smooth one-variable functions that can be used to test whether $\bar{x}$ is a local minimizer of $f.$ Inspired by this result, under the assumption that $f$ is a semi-algebraic function, we construct a finite sequence of real numbers, say, $\{a_1, \ldots, a_p\},$ so that the following statements hold:
\begin{itemize}
\item the point $\bar{x}$ is a local minimizer of $f$ if and only if $a_k \ge 0$ for all $k = 1, \ldots, p;$
\item the point $\bar{x}$ is an isolated local minimizer of $f$ if and only if $a_k > 0$ for all $k = 1, \ldots, p.$
\end{itemize}

It is essential to mention that there is no gap between these necessary and sufficient conditions. Furthermore, the sequence $\{a_1, \ldots, a_p\}$ does not invoke any second-order subdifferential of $f.$ In fact, as we can see in Sections~\ref{Section3} and \ref{Section4}, this sequence is constructed based on the so-called {\em tangency variety} of $f$ at $\bar{x}$ which is defined purely in subdifferential terms. Moreover, in the case where $\bar{x}$ is an isolated local minimizer of $f,$ we determine a ``tangency exponent'' $\alpha_* >0$ such that for all $\alpha \in \mathbb{R}$ the following two statements are equivalent:
\begin{itemize}
\item the inequality $\alpha \ge \alpha_*$ is valid;
\item the point $\bar{x}$ is an {\em $\alpha$th order sharp local minimizer} of $f.$ 
\end{itemize}

The latter means that there exist constants $c > 0$ and $\epsilon > 0$ such that
\begin{eqnarray*}
f(x) &\ge& f(\bar{x}) + c \|x - \bar{x}\|^\alpha \quad \textrm{ for all } \quad x \in \mathbb{B}_{\epsilon}(\bar{x}).
\end{eqnarray*}
It is well-known that second-order growth conditions (i.e., the case of $\alpha = 2$) play an important role in nonlinear optimization, both for convergence analysis of algorithms and for perturbation theory (see, for example, \cite{Bonnas2000, Nocedal2006, Ruszczynski2006}). Under the assumptions that $f$ is convex and $\bar{x}$ is a (necessarily isolated) local minimizer of $f,$ Arag\'on-Artacho and Geoffroy~\cite{AragonArtacho2008} first proved that $\bar{x}$ is a second order sharp local minimizer of $f$ if and only if the limiting subdifferential $\partial f$ is {\em strongly metrically subregular} at $\bar{x}$ for $0$ in the sense that there exist constants $c > 0$ and $\epsilon > 0$ such that
\begin{eqnarray} \label{Eq01}
\mathfrak{m}_f(x) &\ge& c\, \|x - \bar{x}\| \quad \textrm{ for all } \quad x \in \mathbb{B}_{\epsilon}(\bar{x}),
\end{eqnarray}
where $\mathfrak{m}_f(x)$ denotes the minimal norm of subgradients $v \in \partial f(x).$ Afterwards, relaxing the convexity of $f$ to the assumption that $f$ is semi-algebraic, Drusvyatskiy and Ioffe \cite{Drusvyatskiy2015} proved that the corresponding equivalence still holds. Furthermore, they show that if $\bar{x}$ is a (not necessarily isolated) local minimizer, the existence of constants $c > 0$ and $\epsilon > 0$ such that
\begin{eqnarray*}
\mathfrak{m}_f(x) &\ge& c \, \mathrm{dist}\big(x, (\partial f)^{-1}(0)\big) \quad \textrm{ for all } \quad x \in \mathbb{B}_{\epsilon}(\bar{x})
\end{eqnarray*}
implies the existence of constants $c' > 0$ and $\epsilon' > 0$ satisfying
\begin{eqnarray*}
f(x) &\ge& f(\bar{x}) + c'\, \mathrm{dist}\big(x, (\partial f)^{-1}(0)\big)^{2} \quad \textrm{ for all } \quad x \in \mathbb{B}_{\epsilon'}(\bar{x}),
\end{eqnarray*}
where $\mathrm{dist}\big(x, (\partial f)^{-1}(0)\big)$ denotes for the Euclidean distance from $x$ to $(\partial f)^{-1}(0).$ 
In \cite[Remark~3.4]{Drusvyatskiy2014},  the authors conjecture that the converse is also true. We provide a counterexample to this conjecture; see Example~\ref{Counterexample}.

Replacing $\| x - \bar{x}\|$ in \eqref{Eq01} by $\| x - \bar{x}\|^\beta$ with some constant $\beta \in \mathbb{R},$ one can consider the following {\em $\beta$th order strong metric subregularity} of $\partial f$ at $\bar{x}$ for $0$: there exist constants $c > 0$ and $\epsilon > 0$ such that
\begin{eqnarray*}
\mathfrak{m}_f(x) &\ge& c\, \|x - \bar{x}\|^{\beta} \quad \textrm{ for all } \quad x \in \mathbb{B}_{\epsilon}(\bar{x})  \setminus \{\bar{x}\}.
\end{eqnarray*}
(Note that we exclude $\bar{x}$ here because $\beta$ may be negative; for example, the limiting subdifferential of the continuous function $\mathbb{R} \rightarrow \mathbb{R}, x \mapsto \sqrt{|x|},$ is strongly metrically subregular of order $\beta = -\frac{1}{2}$ at $\bar{x} = 0$ for $0$). Metric regularity and (strong) metric subregularity are becoming an important and active area of research in variational analysis and optimization theory. For more details, we refer the reader to the books \cite{Dontchev2009, Klatte2002, Mordukhovich2006} and the survey 
 \cite{Ioffe2016-1, Ioffe2016-2} with references therein. Recently, under the assumptions that $f$ is convex, $\bar{x}$ is a local minimizer of $f,$ and that $\alpha > 1,$ Zheng and Ng \cite{Zheng2015} and, independently, Mordukhovich and Ouyang \cite{Mordukhovich2015}  showed that $\bar{x}$ is an $\alpha$th order sharp local minimizer of $f$ if and only if the limiting subdifferential $\partial f$ is $(\alpha - 1)$th order strong metric subregularity at $\bar{x}$ for $0.$

In a difference line of development, Bolte, Daniilidis, and Lewis \cite{Bolte2006} showed that if $f$ is subanalytic and $\bar{x}$ is a critical point of $f$ (i.e., $\frak{m}_f(\bar{x}) = 0$), then $f$ satisfies the {\em \L ojaseiwcz gradient inequality} at $\bar{x}$ with an exponent $\theta \in [0, 1),$ which means that there exist constants $c > 0$ and $\epsilon > 0$ such that
\begin{eqnarray*}
\mathfrak{m}_f(x) &\ge& c\, |f(x) - f(\bar{x})|^{\theta}  \quad \textrm{ for all } \quad x \in \mathbb{B}_{\epsilon}(\bar{x}) \setminus \{\bar{x}\}.
\end{eqnarray*}
It is worth emphasizing that the convergence behavior of many first-order methods can be understood using the \L ojasiewicz gradient inequality and its associated exponent; see, for example, \cite{Absil2005, Attouch2010, Attouch2013, Bolte2014, Bolte2017, Frankel2015, Li2015, Li2016, Li2018-2}.

Motivated by the aforementioned works, we show that if $f$ is semi-algebraic and $\bar{x}$ is an isolated local minimizer of $f,$ then for any $\alpha \ge \alpha_*,$ the following statements are equivalent:
\begin{itemize}
\item The point $\bar{x}$ is an $\alpha$th order sharp local minimizer of $f.$
\item The limiting subdifferential $\partial f$ is $(\alpha - 1)$th order strongly metrically subregular at $\bar{x}$ for $0.$
\item The function $f$ satisfies the \L ojaseiwcz gradient inequality at $\bar{x}$ with the exponent $1 - \frac{1}{\alpha}.$ 
\end{itemize}

Note that, for a special value of $\alpha,$ these three equivalences were proved by Gwo\'zdziewicz \cite{Gwozdziewicz1999} (with $f$ being an analytic function) and by the author \cite{PHAMTS2011} (with $f$ being a continuous subanalytic function).

To be concrete, we study only semi-algebraic functions. Analogous results, with essentially identical proofs, also hold for  functions definable in a polynomially bounded o-minimal structure (see \cite{Dries1996} for more on the subject).  However, to lighten the exposition, we do not pursue this extension here.

The rest of this paper is organized as follows. Section~\ref{SectionPreliminary} contains some preliminaries from variational analysis and semi-algebraic geometry widely used in the proofs of the main results given below. The tangency variety, which plays an important role in this study, is presented in Section~\ref{Section3}.
The main results are given in Section~\ref{Section4}. Finally, several examples are provided in Section~\ref{Section5}.

\section{Preliminaries} \label{SectionPreliminary}

Throughout this work we shall consider the Euclidean vector space ${\Bbb R}^n$ endowed with its canonical scalar product $\langle \cdot, \cdot \rangle,$ and we shall denote its associated norm $\| \cdot \|.$ The closed ball (resp., the sphere) centered at $\bar{x} \in \mathbb{R}^n$ of radius $\epsilon$ will be denoted by $\mathbb{B}_{\epsilon}(\bar{x})$ (resp., $\mathbb{S}_{\epsilon}(\bar{x})$). When $\bar{x}$ is the origin of $\mathbb{R}^n$ we write $\mathbb{B}_{\epsilon}$ instead of $\mathbb{B}_{\epsilon}(\bar{x}).$ 

For a function $f \colon \mathbb{R}^n \rightarrow {\mathbb{R}},$ we define the {\em epigraph} of $f$ to be
\begin{eqnarray*}
\mathrm{epi} f &:=& \{ (x, y) \in \mathbb{R}^n \times \mathbb{R} \ | \ f(x) \le y \}.
\end{eqnarray*}
A function $f  \colon \mathbb{R}^n \rightarrow {\mathbb{R}}$ is said to be {\em lower semi-continuous} if for each $x \in \mathbb{R}^n$ the inequality $\liminf_{x' \to {x}} f(x') \ge f({x})$ holds.

\subsection{Normals and subdifferentials}
Here we recall the notions of the normal cones to sets and the subdifferentials of real-valued functions used in this paper. The reader is referred to \cite{Mordukhovich2006, Mordukhovich2018, Rockafellar1998} for more details.

\begin{definition}{\rm Consider a set $\Omega\subset\mathbb{R}^n$ and a point ${x} \in \Omega.$
\begin{enumerate}
\item[(i)] The {\em regular normal cone} (known also as the {\em prenormal} or {\em Fr\'echet normal cone}) $\widehat{N}({x}; \Omega)$ to
$\Omega$ at ${x}$ consists of all vectors $v\in\mathbb{R}^n$ satisfying
\begin{eqnarray*}
\langle v, x' - {x} \rangle &\le& o(\|x' -  {x}\|) \quad \textrm{ as } \quad x' \to {x} \quad \textrm{ with } \quad x' \in \Omega.
\end{eqnarray*}

\item[(ii)] The {\em limiting normal cone} (known also as the {\em basic} or {\em Mordukhovich normal cone}) $N({x}; \Omega)$ to $\Omega$ at ${x}$ consists of all vectors $v \in \mathbb{R}^n$ such that there are sequences $x^k \to {x}$ with $x^k \in \Omega$ and $v^k \rightarrow v$ with $v^k \in \widehat N(x^k; \Omega).$
\end{enumerate}
}\end{definition}

If $\Omega$ is a manifold of class $C^1,$ then for every point $x \in \Omega,$ the normal cones $\widehat{N}({x}; \Omega)$ and $N({x}; \Omega)$ are equal to the normal space to $\Omega$ at ${x}$ in the sense of differential geometry; see \cite[Example~6.8]{Rockafellar1998}.

\begin{definition}{\rm
Consider a function $f\colon\mathbb{R}^n \to {\mathbb{R}}$ and a point ${x} \in \mathbb{R}^n.$
\begin{enumerate}
\item[(i)] The {\em limiting} and {\em horizon subdifferentials} of $f$ at ${x}$ are defined, respectively, by
\begin{eqnarray*}
\partial f({x}) &:=& \big\{v\in\mathbb{R}^n\;\big|\;(v,-1)\in N\big(({x}, f({x}));\mathrm{epi} f\big)\big\}, \\
\partial^\infty f({x}) &:=& \big\{v\in\mathbb{R}^n\;\big|\;(v,0)\in N\big(({x}, f({x}));\mathrm{epi} f \big)\big\}.
\end{eqnarray*}

\item[(ii)] The {\em nonsmooth slope} of $f$ at ${x}$ is defined by
\begin{eqnarray*}
{\frak m}_f({x}) &:=& \inf \{ \|v\| \ | \ v \in {\partial} f({x}) \}.
\end{eqnarray*}
By definition, ${\frak m}_f({x}) = + \infty$ whenever ${\partial} f({x}) = \emptyset.$
\end{enumerate}
}\end{definition}

In \cite{Mordukhovich2006, Mordukhovich2018, Rockafellar1998} the reader can find equivalent analytic descriptions of the limiting subdifferential $\partial f({x})$ and comprehensive studies of it and related constructions. For convex $f,$ this subdifferential coincides with the convex subdifferential. Furthermore, if the function $f$ is of class $C^1,$ then $\partial f({x}) = \{\nabla f({x})\}$ and so ${\frak m}_f(x) = \|\nabla f(x) \|.$ The horizon subdifferential $\partial^\infty f({x})$ plays an entirely different role--it detects horizontal ``normal'' to the epigraph--and it plays a decisive role in subdifferential calculus; see \cite[Corollary~10.9]{Rockafellar1998} for more details.

\begin{theorem}[Fermat rule]\label{FermatRule}
Consider a lower semi-continuous function $f \colon\mathbb{R}^n \to {\mathbb{R}}$ and a closed set $\Omega \subset \mathbb{R}^n.$ If $\bar{x} \in \Omega$ is a local minimizer of $f$ on $\Omega$ and the qualification condition
\begin{eqnarray*}
\partial^\infty f(\bar{x}) \cap N(\bar{x}; \Omega) &=& \{0\}
\end{eqnarray*}
is valid, then the inclusion $0 \in \partial f(\bar{x}) + N(\bar{x}; \Omega)$ holds.
\end{theorem}

\subsection{Semi-algebraic geometry}

Now, we recall some notions and results of semi-algebraic geometry, which can be found in \cite{Bochnak1998, Dries1996}.

\begin{definition}{\rm
A subset $S$ of $\mathbb{R}^n$ is called {\em semi-algebraic} if it is a finite union of sets of the form
$$\{x \in \mathbb{R}^n \ | \  f_i(x) = 0, \ i = 1, \ldots, k; f_i(x) > 0, \ i = k + 1, \ldots, p\},$$
where all $f_{i}$ are polynomials.
In other words, $S$ is a union of finitely many sets, each defined by finitely many polynomial equalities and inequalities.
A function $f \colon S \rightarrow {\mathbb{R}}$ is said to be {\em semi-algebraic} if its graph
\begin{eqnarray*}
\{ (x, y) \in S \times \mathbb{R} \ | \ y =  f(x) \}
\end{eqnarray*}
is a semi-algebraic set.
}\end{definition}

A major fact concerning the class of semi-algebraic sets is its stability under linear projections (see, for example, \cite{Bochnak1998}).

\begin{theorem}[Tarski--Seidenberg Theorem] \label{TarskiSeidenbergTheorem}
The image of any semi-algebraic set $S \subset \mathbb{R}^n$ under a projection to any linear subspace of $\mathbb{R}^n$ is a semi-algebraic set.
\end{theorem}

\begin{remark}{\rm
As an immediate consequence of the Tarski--Seidenberg Theorem, we get semi-algebraicity of any set $\{ x \in A : \exists y \in B,  (x, y) \in C \},$  provided that $A ,  B,$  and $C$  are semi-algebraic sets in the corresponding spaces. Also, $\{ x \in A : \forall y \in B,  (x, y) \in C \}$ is a semi-algebraic set as its complement is the union of the complement of $A$  and the set $\{ x \in A : \exists y \in B,  (x, y) \not\in C \}.$ Thus, if we have a finite collection of semi-algebraic sets, then any set obtained from them with the help of a finite chain of quantifiers is also semi-algebraic. In particular, for a semi-algebraic function $f \colon \mathbb{R}^n \rightarrow {\mathbb{R}},$ it is easy to see that the nonsmooth slope $\frak{m}_f \colon \mathbb{R}^n \rightarrow {\mathbb{R}}$ is a semi-algebraic function. 
}\end{remark}

The following three well-known lemmas will be of great importance for us; see, for example, \cite[Theorem~1.8, Theorem~1.11, and Lemma~1.7]{HaHV2017}.

\begin{lemma} [Monotonicity Lemma] \label{MonotonicityLemma}
Let $f \colon (a, b) \rightarrow \mathbb{R}$ be a semi-algebraic function. Then there are finitely many points $a = t_0 < t_1 < \cdots < t_k = b$ such that the restriction of $f$ to each interval $(t_i, t_{i + 1})$ is analytic, and either constant, or strictly increasing or strictly decreasing.
\end{lemma}

\begin{lemma}[Curve Selection Lemma] \label{CurveSelectionLemma}
Consider a semi-algebraic set $S \subset \mathbb{R}^n$ and a point $\bar{x} \in \mathbb{R}^n$ that is a cluster point of $S.$ Then there exists an analytic semi-algebraic curve $\phi \colon (0, \epsilon) \to {\mathbb R}^n$ with $\lim_{t \to 0^+} \phi(t) = \bar{x}$ and with $\phi(t) \in S$ for $t \in (0, \epsilon).$
\end{lemma}

\begin{lemma}[Growth Dichotomy Lemma] \label{GrowthDichotomyLemma}
Let $f \colon (0, \epsilon) \rightarrow {\mathbb R}$ be a semi-algebraic function with $f(t) \ne 0$ for all $t \in (0, \epsilon).$ Then there exist constants $a \ne 0$ and $\alpha \in {\mathbb Q}$ such that $f(t) = at^{\alpha} + o(t^{\alpha})$ as $t \to 0^+.$
\end{lemma}

In the sequel we will make use of Hardt's semi-algebraic triviality. We present a particular case--adapted to our needs--of a more general result: see \cite{Bochnak1998, Hardt1980, Dries1996} for the statement in its full generality.

\begin{theorem}[Hardt's semi-algebraic triviality] \label{HardtTheorem}
Let $S$ be a semi-algebraic set in $\mathbb{R}^n$ and $f \colon S \rightarrow\mathbb{R}$ a continuous semi-algebraic map.
Then there are finitely many points $-\infty = t_0 < t_1 < \cdots < t_k = +\infty$ such that $f$ is semi-algebraically trivial over each the interval $(t_i, t_{i + 1}),$ that is, there exists a semi-algebraic set $F_i \subset \mathbb{R}^n$ and a semi-algebraic homeomorphism
$h_i \colon f^{-1} (t_i, t_{i + 1}) \rightarrow (t_i, t_{i + 1}) \times F_i$ such that the composition $h_i$ with the projection $(t_i, t_{i + 1}) \times F_i \rightarrow (t_i, t_{i + 1}), (t, x) \mapsto t,$  is equal to the restriction of $f$ to $f^{-1} (t_i, t_{i + 1}).$
\end{theorem}

We will also need the following lemma. 

\begin{lemma} \label{Lemma210}
Consider a lower semi-continuous semi-algebraic function $f \colon \mathbb{R}^n \rightarrow {\mathbb{R}}$ and a semi-algebraic curve $\phi \colon [a, b] \rightarrow \mathbb{R}^n.$ Then for all but finitely many ${{t}} \in [a, b],$ the mappings $\phi$ and $f \circ \phi$ are analytic at ${{t}}$ and satisfy
\begin{eqnarray*}
v \in \partial f (\phi({{t}})) &\Longrightarrow& \langle v, \dot{\phi}({{t}}) \rangle \ = \ (f \circ \phi)'({{t}}), \\
v \in \partial^\infty f (\phi({{t}})) &\Longrightarrow& \langle v, \dot{\phi}({{t}}) \rangle \ = \ 0.
\end{eqnarray*}
\end{lemma}

\begin{proof}
This follows immediately from \cite[Lemma~2.10]{Drusvyatskiy2015} (see also \cite[Proposition~4]{Bolte2007-2}), and so is omitted.
\end{proof}

\section{Tangencies} \label{Section3}

From now on, let $f \colon \mathbb{R}^n \rightarrow {\Bbb R}$ be a non-constant semi-algebraic function, which is continuous around a point $\bar{x} \in \mathbb{R}^n.$ Using the so--called tangency variety of $f$ at $\bar{x},$ we define finite sets of real numbers that can be used to test if $f$ has a local minimizer at $\bar{x}$ and if $f$ has an $\alpha$th order sharp local minimizer at $\bar{x}.$ Let us begin with the following definition (see also \cite{HaHV2017}).

\begin{definition}{\rm
The {\em tangency variety of $f$ at $\bar{x}$} is defined as follows:
\begin{eqnarray*}
\Gamma(f) &:=& \{x \in \mathbb{R}^n  \ | \ \exists \lambda \in {\Bbb R} \textrm{ such that } \lambda (x - \bar{x}) \in \partial f(x) \}.
\end{eqnarray*}
}\end{definition}

Remark that under mild regularity conditions, $\Gamma(f)$ is the set of critical points of the function $f + \delta_{{\Bbb B}_{t}(\bar{x})},$ 
where $\delta_{{\Bbb B}_{t}(\bar{x})}$ denotes the {\em indicator function} of the ball ${\Bbb B}_{t}(\bar{x}).$ Moreover, thanks to the Fermat rule (Theorem~\ref{FermatRule}), we can see that for all sufficiently small $t > 0,$ the tangency variety $\Gamma(f)$ contains the set of minimizers of the optimization problem $\min_{x \in {\Bbb S}_{t}(\bar{x})} f(x);$ in particular, $\bar{x}$ is a cluster point of $\Gamma(f).$

By the Tarski--Seidenberg Theorem~\ref{TarskiSeidenbergTheorem}, $\Gamma(f)$ is a semi-algebraic set. Applying Hardt's triviality Theorem~\ref{HardtTheorem} for the continuous semi-algebraic function
$$\rho \colon \Gamma(f) \rightarrow \mathbb{R}, \quad x \mapsto \|x - \bar{x}\|,$$
we get a semi-algebraic set $F \subset \mathbb{R}^n$ and a semi-algebraic homeomorphism
$$h \colon \rho^{-1} \big((0, \epsilon)\big) \rightarrow (0, \epsilon) \times F$$ 
such that the following diagram commutes:
\begin{equation*}
\CD \rho^{-1}\big((0, \epsilon)\big) @> h >> (0, \epsilon)  \times F \\
@V \rho VV @V \pi VV\\
(0, \epsilon)   @> \mathrm{id} >> (0, \epsilon) 
\endCD
\end{equation*}
where $\pi$ is the projection on the first component of the product and $\mathrm{id}$ is the identity map.

Since $F$ is semi-algebraic, the  number of its connected components, say, $p,$ is finite. Then $\Gamma(f) \cap \mathbb{B}_\epsilon(\bar{x}) \setminus \{\bar{x}\}$ has exactly $p$ connected components, say, $\Gamma_1, \ldots, \Gamma_p,$ and each such component is a semi-algebraic set. Moreover, for all $t \in (0, \epsilon)$ and all $k = 1, \ldots, p,$ the sets $\Gamma_k \cap \mathbb{S}_t(\bar{x})$ are connected (recall that $\mathbb{S}_t(\bar{x})$ stands for the sphere centered at $\bar{x}$ of radius $t$). Corresponding to each $\Gamma_k,$ let
$$f_k \colon (0, \epsilon) \rightarrow \mathbb{R}, \quad t \mapsto f_k(t),$$
be the function defined by $f_k(t) :=  f(x),$ where $x \in \Gamma_k \cap \mathbb{S}_t(\bar{x}).$

\begin{lemma}
For all $\epsilon > 0$ small enough, the following statements hold:
\begin{enumerate}
\item[{\rm (i)}]  All the functions $f_k$ are well-defined and semi-algebraic.
\item[{\rm (ii)}] Each the function $f_k$ is either constant or strictly monotone.
\end{enumerate}
\end{lemma}

\begin{proof}
(i) Fix $k \in \{1, \ldots, p\},$ and take any $t \in (0, \epsilon).$ We will show that the restriction of $f$ on $\Gamma_k \cap \mathbb{S}_t(\bar{x})$ is constant. To see this, let $\phi \colon [0, 1] \rightarrow {\Bbb R}^n$ be a smooth semi-algebraic curve such that $\phi(\tau) \in \Gamma_k \cap \mathbb{S}_t(\bar{x})$ for all $\tau \in [0, 1].$ By definition, we have
\begin{eqnarray*}
\|\phi(\tau) - \bar{x}\|  = t \quad & \textrm{ and } & \quad \lambda(\tau) (\phi(\tau) - \bar{x}) \in \partial f(\phi(\tau))
\end{eqnarray*}
for some $\lambda(\tau) \in \mathbb{R}.$ Moreover, in view of Lemma~\ref{Lemma210}, for all but finitely many $\tau \in [a, b],$ the mappings $\phi$ and $f \circ \phi$ are analytic at $\tau$ and satisfy
\begin{eqnarray*}
v \in \partial f (\phi(\tau)) &\Longrightarrow& \langle v, \dot{\phi}(\tau) \rangle \ = \ (f \circ \phi)'(\tau).
\end{eqnarray*}
Therefore
\begin{eqnarray*}
(f \circ \phi)'(\tau) &=&  \langle \lambda(\tau) \big(\phi(\tau) - \bar{x}\big), \dot{\phi}(\tau) \rangle \\
&=& \frac{\lambda(\tau)}{2} \frac{d \|\phi(\tau) - \bar{x}\|^2}{d\tau}\\
&=& 0.
\end{eqnarray*}
So $f$ is constant on the curve $\phi.$

On the other hand, since the set $\Gamma_k \cap \mathbb{S}_t(\bar{x})$ is connected semi-algebraic, it is path connected. Hence, any two points in $\Gamma_k \cap \mathbb{S}_t(\bar{x})$ can be joined by a piecewise smooth semi-algebraic curve (see \cite[Theorem~1.13]{HaHV2017}). It follows that  the restriction of $f$ on $\Gamma_k \cap \mathbb{S}_t(\bar{x})$ is constant and so the function $f_k$ is well-defined. Finally, by the Tarski--Seidenberg Theorem~\ref{TarskiSeidenbergTheorem}, $f_k$ is semi-algebraic.

(ii) This is a direct consequence of Lemma~\ref{MonotonicityLemma} (perhaps after reducing $\epsilon$).
\end{proof}

For each $t \in (0, \epsilon),$ the sphere $\mathbb{S}_t(\bar{x})$ is a nonempty compact semi-algebraic set. Hence, the function
$$\psi \colon (0, \epsilon) \rightarrow \mathbb{R}, \quad t \mapsto \psi(t) := \min_{x \in \mathbb{S}_t(\bar{x})} f(x),$$
is well-defined, and moreover, it is semi-algebraic because of the Tarski--Seidenberg Theorem~\ref{TarskiSeidenbergTheorem} (see the discussion in \cite[Section~1.6]{HaHV2017}).  The following lemma is simple but useful.
\begin{lemma} \label{Lemma32}
For $\epsilon > 0$ small enough, the following statements hold:
\begin{enumerate}
\item [{\rm (i)}] The functions $\psi$ and $f_1, \ldots, f_p$ are either coincide or disjoint.

\item [{\rm (ii)}] $\psi(t) = \min_{k = 1, \ldots, p} f_k(t)$ for all $t \in (0, \epsilon).$

\item [{\rm (iii)}] There exists an index $k \in \{1, \ldots, p\}$ such that $\psi(t) = f_k(t)$ for all $t \in (0, \epsilon).$
\end{enumerate}
\end{lemma}

\begin{proof}
(i) This is an immediate consequence of the Monotonicity Lemma~\ref{MonotonicityLemma}.

(ii) Without loss of generality, assume $\bar{x} = 0$ and $f (\bar{x}) = 0.$ 
Applying the Curve Selection Lemma~\ref{CurveSelectionLemma} and shrinking $\epsilon$ (if necessary), we find an analytic semi-algebraic curve $\phi \colon (0, \epsilon) \rightarrow \mathbb{R}^n$ such that $\|\phi(t)\| = t$ and $f \circ \phi(t) = \psi(t)$ for all $t.$ By Lemma~\ref{Lemma210}, then we have for any $t \in (0, \epsilon),$
\begin{eqnarray}
v \in \partial^\infty f(\phi(t)) & \Longrightarrow & \langle v, \dot{\phi}(t) \rangle \ = \ 0 \nonumber.
\end{eqnarray}
Observe
\begin{eqnarray*}
\langle \phi(t), \dot{\phi}(t) \rangle & = & \frac{1}{2} \frac{d}{d t} \|\phi(t)\|^2,
\end{eqnarray*}
and hence the qualification condition
\begin{eqnarray*}
\partial^\infty f(\phi(t)) \cap N\big(\phi(t); \mathbb{S}_t(\bar{x})\big) = \{0\}
\end{eqnarray*}
holds for all $t \in (0, \epsilon).$ Consequently, since $\phi(t)$ minimizes $f$ subject to $\|x\| = t,$
applying the Fermat rule (Theorem~\ref{FermatRule}), we deduce that $\phi(t)$ belongs to $\Gamma(f).$ Therefore,
\begin{eqnarray*}
\psi(t)  &=& \min_{x \in \mathbb{S}_t(\bar{x})} f(x) \ = \ \min_{x \in \Gamma(f) \cap \mathbb{S}_t(\bar{x})} f(x) \ = \ \min_{k = 1, \ldots, p} \min_{x \in \Gamma_k \cap \mathbb{S}_t(\bar{x})} f(x) \ = \ \min_{k = 1, \ldots, p}  f_k(t),
\end{eqnarray*}
where the third equality follows from the fact that
\begin{eqnarray*}
\Gamma(f) \cap \mathbb{S}_t(\bar{x}) &=& \bigcup_{k = 1}^p \Gamma_k \cap \mathbb{S}_t(\bar{x}).
\end{eqnarray*}

(iii) This follows from items (i) and (ii).
\end{proof}

\section{Main results} \label{Section4}

Recall that $f \colon \mathbb{R}^n \rightarrow {\Bbb R}$ is a non-constant semi-algebraic function, which is continuous around a point $\bar{x} \in \mathbb{R}^n.$ As in the previous section, we associate to the function $f$ a finite number of functions $f_1, \ldots, f_p$ of a single variable. Let 
$$K :=\{k \ | \ f_k \textrm{ is not constant} \}.$$
Note that $f_k \equiv f(\bar{x})$ for all $k \not \in K.$ By the Growth Dichotomy Lemma~\ref{GrowthDichotomyLemma}, we can write for each $k \in K,$
\begin{eqnarray*}
f_k(t) &=& f(\bar{x}) + a_k t^{\alpha_k} + o(t^{\alpha_k}) \quad \textrm{ as } \quad t \to 0^+,
\end{eqnarray*}
where $a_k \in \mathbb{R}, a_k \ne 0,$ and $\alpha_k \in \mathbb{Q}, \alpha_k > 0.$ It is convenient to define $a_k = 0$ for $k \not \in K.$
As we can see the ``tangency coefficients'' $a_k$ and the ``tangency exponents'' $\alpha_k$ play important roles in Theorems~\ref{Theorem41} and \ref{Theorem42} below.

We now arrive to the first main result of this section. This result provides necessary and sufficient conditions for optimality of nonsmooth semi-algebraic functions.

\begin{theorem}[Necessary and sufficient conditions for optimality] \label{Theorem41}
With the above notations, the following statements hold:
\begin{enumerate}
\item [{\rm (i)}] The point $\bar{x}$ is a local minimizer of $f$ if and only if $a_k \ge 0$ for all $k = 1, \ldots, p.$

\item [{\rm (ii)}] The point $\bar{x}$ is an isolated local minimizer of $f$ if and only if $a_k > 0$ for all $k = 1, \ldots, p.$
\end{enumerate}
\end{theorem}
\begin{proof}
Recall that
\begin{eqnarray*}
\psi(t) &:=& \min_{x \in \mathbb{S}_t(\bar{x})} f(x) \quad \textrm{ for } \quad t \ge 0.
\end{eqnarray*}
By definition, it is easy to see that $\bar{x}$ is a local minimizer (resp., an isolated local minimizer) of $f$ if and only if for all $t > 0$ small enough, we have
$\psi(t) \ge f(\bar{x})$ (resp., $\psi(t) > f(\bar{x})$). This observation, together with Lemma~\ref{Lemma32}, implies easily the desired conclusion.
\end{proof}

\begin{remark}{\rm
As shown in Section~\ref{Section5} below, when the tangency variety $\Gamma(f)$ is an algebraic curve, the numbers $a_k$ and $\alpha_k$ can be computed using algebraic methods. Very recently, using tangency varieties, Guo and Ph\d{a}m \cite{Guo2018} proposed a computational and symbolic algorithm to determine the type (local minimizer, local maximizer, or saddle point) of a given isolated critical point, which is degenerate, of a multivariate polynomial function. So it is our hope that in the general case, there are algorithms to compute the numbers $a_k$ and $\alpha_k,$ and this will be studied in future work.
}\end{remark}

We know from \L ojasiewicz's inequality \cite[Theorem~1.14]{HaHV2017} that $\bar{x}$ is an isolated local minimizer of $f$ if and only if there exists a real number $\alpha > 0$ such that $\bar{x}$ is an $\alpha$th order sharp local minimizer of $f.$
A characteristic of this number $\alpha$ in terms of the ``tangency exponents'' of $f$ is given in Theorem~\ref{Theorem42} below. To this end, let
\begin{eqnarray*}
\alpha_* &:=& \max_{k \in K} \alpha_k \ > \ 0.
\end{eqnarray*}
The second main result of this section reads as follows.

\begin{theorem}[Isolated local minimizers] \label{Theorem42}
With the above notations, assume that $\bar{x} \in \mathbb{R}^n$ is an isolated local minimizer of $f.$ Then 
for any $\alpha \in \mathbb{R},$ the following statements are equivalent:
\begin{enumerate}
\item[{\rm (i)}] The inequality $\alpha \ge \alpha_*$ holds.

\item[{\rm (ii)}] The point $\bar{x}$ is an $\alpha$th order sharp local minimizer of $f,$ i.e., there exist constants $c > 0$ and $\epsilon > 0$ such that
\begin{eqnarray*}
f(x) &\ge& f(\bar{x}) + c\, \|x - \bar{x}\|^\alpha \quad \textrm{ for all } \quad x \in \mathbb{B}_{\epsilon}(\bar{x}).
\end{eqnarray*}

\item[{\rm (iii)}] The limiting subdifferential $\partial f$ of $f$ is $(\alpha - 1)$th order strongly metrically subregular at $\bar{x}$ for $0,$ i.e.,  there exist constants $c > 0$ and $\epsilon > 0$ such that
\begin{eqnarray*}
\mathfrak{m}_f(x) &\ge& c\, \|x - \bar{x}\|^{\alpha - 1} \quad \textrm{ for all } \quad x \in \mathbb{B}_{\epsilon}(\bar{x}) \setminus \{\bar{x}\}.
\end{eqnarray*}

\item[{\rm (iv)}] The function $f$ satisfies the \L ojaseiwcz gradient inequality at $\bar{x}$ with the exponent $1 - \frac{1}{\alpha},$ i.e., there exist constants $c > 0$ and $\epsilon > 0$ such that
\begin{eqnarray*}
\mathfrak{m}_f(x) &\ge& c\, |f(x) - f(\bar{x})|^{1 - \frac{1}{\alpha}}  \quad \textrm{ for all } \quad x \in \mathbb{B}_{\epsilon}(\bar{x}) \setminus \{\bar{x}\}.
\end{eqnarray*}

\end{enumerate}
\end{theorem}

In order to prove Theorem~\ref{Theorem42} below, we need the following result which can be seen as a nonsmooth version of the Bochnack--\L ojasiewicz inequality \cite{Bochnak1971}.

\begin{lemma} \label{Lemma41}
There exist constants $c > 0$ and $\epsilon > 0$ such that
\begin{eqnarray*}
\mathfrak{m}_f (x) \| x - \bar{x}\| & \ge & c |f(x) -  f(\bar{x}) | \quad \textrm{ for all } \quad x \in \mathbb{B}_{\epsilon}(\bar{x}).
\end{eqnarray*}
\end{lemma}
\begin{proof}
Without loss of generality, we may assume that $\bar{x} = 0$ and $f(\bar{x}) = 0.$

Arguing by contradiction, suppose that the lemma is false, that is,
\begin{eqnarray*}
\liminf_{x \to \bar{x}} \frac{\mathfrak{m}_f(x)\|x\|}{|f(x)|} &=& 0.
\end{eqnarray*}
In light of the Curve Selection Lemma~\ref{CurveSelectionLemma}, we find a non-constant analytic semi-algebraic curve $\phi \colon (0, \epsilon) \rightarrow \mathbb{R}^n$ with $\lim_{{{t}} \to 0^+} \phi({{t}}) = 0$ such that $f \circ \phi({{t}}) \ne 0$ and
\begin{eqnarray*}
\lim_{{{t}} \to 0^+} \frac{\mathfrak{m}_f (\phi({{t}})) \| \phi({{t}})\|}{|f \circ \phi({{t}})|} &=& 0.
\end{eqnarray*}
Since $f$ is continuous at $\bar{x},$ it holds that
\begin{eqnarray*}
\lim_{{{t}} \to 0^+} f \circ \phi({{t}}) &=& 0.
\end{eqnarray*}
By the Growth Dichotomy Lemma~\ref{GrowthDichotomyLemma}, we can write
\begin{eqnarray*}
\phi({{t}}) & = & a {{t}}^\alpha + o({{t}}^\alpha) \quad \textrm{ and } \quad f \circ \phi({{t}}) \  = \ b {{t}}^\beta + o({{t}}^\beta) \quad \textrm{ as } {{t}} \to 0^+,
\end{eqnarray*}
for some $a \in \mathbb{R}^n, a \ne 0, \alpha \in \mathbb{Q}, \alpha > 0, b \in \mathbb{R}, b \ne 0,$ and $\beta \in \mathbb{Q}, \beta > 0.$ It follows that
\begin{eqnarray*}
\dot{\phi}({{t}}) & = & \alpha a {{t}}^{\alpha - 1} + o({{t}}^{\alpha - 1}) \quad \textrm{ and } \quad (f \circ \phi)'({{t}}) \  = \ \beta b {{t}}^{\beta - 1} + o({{t}}^{\beta - 1}) \quad \textrm{ as } {{t}} \to 0^+.
\end{eqnarray*}
Then a direct calculation shows that for all sufficiently small ${{t}} > 0,$
\begin{eqnarray*}
\frac{\alpha}{2} \|\phi({{t}})\|   &\le& \| {{t}} \dot{\phi}({{t}})\|  \ \le \ 2{\alpha} \|\phi({{t}})\| , \\  
\frac{\beta}{2} |f \circ \phi({{t}})| & \le &  |{{t}} (f \circ \phi)'({{t}})| \ \le \ 2{\beta} |f \circ \phi({{t}})|.
\end{eqnarray*}
On the other hand,  we deduce easily from Lemma~\ref{Lemma210} that
\begin{eqnarray*}
|(f \circ \phi)'({{t}})| &\le& \mathfrak{m}_f (\phi({{t}}))  \| \dot{\phi}({{t}}) \|.
\end{eqnarray*}
Therefore,
\begin{eqnarray*}
\frac{\beta}{2}  |f \circ \phi({{t}})| 
& \le & |{{t}} (f \circ \phi)'({{t}})|  \ \le \ \mathfrak{m}_f (\phi({{t}}))  \| {{t}} \dot{\phi}({{t}}) \| \ \le \ 2 \alpha\, \mathfrak{m}_f (\phi({{t}}))  \| {\phi}({{t}}) \|.
\end{eqnarray*}
Consequently, 
\begin{eqnarray*}
0 \ < \ \frac{\beta}{4 \alpha} & \le &  \frac{\mathfrak{m}_f (\phi({{t}})) \|\phi({{t}})\|}{|f \circ \phi({{t}})|}
\end{eqnarray*}
for all sufficiently small ${{t}} > 0.$ Letting ${{t}}$ tend to zero in this inequality, we arrive at a contradiction.
\end{proof}

\begin{proof}[Proof of Theorem~\ref{Theorem42}]
Without loss of generality, assume $\bar{x} = 0$ and $f (\bar{x}) = 0.$ 

By Theorem~\ref{Theorem41}, $K = \{1, \ldots, p\}$ and $a_k > 0$ for all $k \in K.$ Recall that
\begin{eqnarray*}
\psi({{t}})  &:=&  \min_{x \in \mathbb{S}_{t} (\bar{x})} f(x).
\end{eqnarray*}
In light of Lemma~\ref{Lemma32}, we can write
\begin{eqnarray}\label{Eq02}
\psi(t) &=& a_* t^{\alpha_*} + o(t^{\alpha_*}) \quad \textrm{ as } \quad t \to 0^+,
\end{eqnarray}
where $a_*  := \min\{a_k \ | \ k \in K \textrm{ and } \alpha_k = \alpha_*\}.$ In particular, for any real number $c  \in (0,  a_*)$ there exists $\epsilon \in (0, 1)$ such that
\begin{eqnarray} \label{Eq03}
\psi(t) &\ge& c\, t^{\alpha_*} \quad \textrm{ for all } \quad t \in [0, \epsilon].
\end{eqnarray}

(i) $\Leftrightarrow$ (ii):
Assume that $\alpha \ge \alpha_*.$ From \eqref{Eq03} we have for all $x \in \mathbb{B}_{\epsilon}(\bar{x}),$
\begin{eqnarray*}
f(x) &\ge& \psi(\|x\|) \ \ge \ c\,\|x\|^{\alpha_*} \ \ge \ c\, \|x\|^{\alpha},
\end{eqnarray*}
which proves (ii).

Conversely, assume that there exist constants $c' > 0$ and $\epsilon' > 0$ such that
\begin{eqnarray*}
f(x) &\ge& c' \|x\|^\alpha \quad \textrm{ for all } \quad x \in \mathbb{B}_{\epsilon'}(\bar{x}).
\end{eqnarray*}
Then for all $t \in [0, \epsilon]$ we have
\begin{eqnarray*}
\psi(t) &=& \min_{x \in \mathbb{S}_{t} (\bar{x})} f(x) \ \ge \ c' t^\alpha.
\end{eqnarray*}
Combining this with \eqref{Eq02} we get $\alpha \ge \alpha_*.$

(iv) $\Rightarrow$ (iii) $\Rightarrow$ (ii): Clearly, the condition (iii) holds provided that both the conditions (ii) and (iv) hold. So it suffices to show the implications (iii) $\Rightarrow$ (ii) and (iv) $\Rightarrow$ (ii). 

Note that the minimum in the definition of $\psi$ is attained. In view of the Curve Selection Lemma~\ref{CurveSelectionLemma}, there is an analytic semi-algebraic curve $\phi \colon (0, \epsilon) \rightarrow \mathbb{R}^n$ such that $\|\phi({{t}})\| = {{t}}$ and $f \circ \phi({{t}}) = \psi({{t}})$ for all ${{t}}.$ Applying Lemma~\ref{Lemma210} and shrinking $\epsilon$ (if necessary), we have for any ${{t}} \in (0, \epsilon),$
\begin{eqnarray}
v \in \partial f(\phi({{t}})) & \Longrightarrow & \langle v, \dot{\phi}({{t}}) \rangle \ = \ \psi'({{t}}), \label{Eq04} \\
v \in \partial^\infty f(\phi({{t}})) & \Longrightarrow & \langle v, \dot{\phi}({{t}}) \rangle \ = \ 0 \nonumber.
\end{eqnarray}
In particular, as in the proof of Lemma~\ref{Lemma32}, we have $\phi(t) \in \Gamma(f),$ i.e.,  there is a real number $\lambda({{t}})$ satisfying
\begin{eqnarray}\label{Eq05}
\lambda({{t}}) \phi({{t}}) &\in& \partial f(\phi({{t}})).
\end{eqnarray}
By definition, then
\begin{eqnarray*}
\|\lambda({{t}}) \phi({{t}})\| &\ge& \mathfrak{m}_f(\phi({{t}})).
\end{eqnarray*}
Furthermore, it follows from \eqref{Eq04} and \eqref{Eq05} that
\begin{eqnarray*}
\psi'({{t}}) &=& \lambda({{t}}) \langle \phi({{t}}), \dot{\phi}({{t}}) \rangle \ = \
\lambda({{t}}) \frac{1}{2} \frac{d}{d {{t}}} \|\phi({{t}})\|^2 \ = \ \lambda({{t}})  {{t}}.
\end{eqnarray*}
Consequently, 
\begin{eqnarray*}
|\psi'({{t}})| &=& |\lambda({{t}})  {{t}}| \ = \ \|\lambda({{t}}) \phi({{t}})\| \ \ge \ \mathfrak{m}_f(\phi({{t}})) .
\end{eqnarray*}
Therefore, if the condition (iii) holds, then $|\psi'({{t}})| \ge c\, {{t}}^{\alpha - 1},$ while if the condition (iv) holds, then $|\psi'({{t}})| \ge c\, (\psi({{t}}))^{1 - \frac{1}{\alpha}};$ in both the cases, we get $\alpha \ge \alpha_*$ and so $\psi({{t}}) \ge c' {{t}}^{\alpha}$ for some constant $c' > 0.$
Therefore the condition (ii) holds.

(ii) $\Rightarrow$ (iv):  By assumption, there exist constants $c > 0$ and $\epsilon > 0$ such that
\begin{eqnarray*}
f(x) &\ge& c\, \|x\|^\alpha \quad \textrm{ for all } \quad x \in \mathbb{B}_{\epsilon}(\bar{x}).
\end{eqnarray*}
On the other hand, applying Lemma~\ref{Lemma41}, we deduce that there exist constants $c' > 0$ and $\epsilon' > 0$ such that
\begin{eqnarray*}
\|x\| \mathfrak{m}_f(x) &\ge& c' |f(x)| \quad \textrm{ for all } \quad x \in \mathbb{B}_{\epsilon'}(\bar{x}).
\end{eqnarray*}
Therefore, the inequality
\begin{eqnarray*}
\left(\frac{1}{c}\, f(x) \right)^{\frac{1}{\alpha}} \mathfrak{m}_f(x) &\ge& c' |f(x)|
\end{eqnarray*}
holds for all $x$ near $\bar{x},$ from which the desired conclusion follows.
\end{proof}

From \cite[Example~3.2]{Drusvyatskiy2015} we know that the implication (ii) $\Rightarrow$ (iii), and hence the implication (ii) $\Rightarrow$ (iv), of Theorem~\ref{Theorem42} may easily fail in absence of continuity. The following example shows that the implication (iii) $\Rightarrow$ (iv) of Theorem~\ref{Theorem42} also may fail in absence of continuity.

\begin{example}{\rm
Consider the lower semi-continuous, semi-algebraic function $f \colon \mathbb{R} \rightarrow \mathbb{R}$ defined by
$$f (x) :=
\begin{cases}
1 + x^{2}  & \textrm{ if }  x < 0,\\
x^{2}  & \textrm{ otherwise.}
\end{cases}$$
Observe that $f$ is not continuous at $\bar{x} = 0$ and that $0$ is a second order sharp local minimizer of $f.$
A simple computation shows that
\begin{eqnarray*}
\frak{m}_f(x) &=& 2|x| \quad \textrm{ for all } \quad x \in \mathbb{R},
\end{eqnarray*}
and so the condition (iii) of Theorem~\ref{Theorem42} holds with $\alpha = 2.$ However, it is easy to check that $f$ does not satisfy the condition (iv) of Theorem~\ref{Theorem42}.
}\end{example}

\begin{remark}{\rm
Consider a lower semicontinuous function $f \colon \mathbb{R}^n \rightarrow {\mathbb{R}},$ which has a (not necessarily isolated) local minimum at $\bar{x} \in \mathbb{R}^n.$ It is well-known (see \cite{AragonArtacho2008, AragonArtacho2014, Drusvyatskiy2014, Drusvyatskiy2015, Mordukhovich2015, Zheng2015}) that the existence of constants $c > 0$ and $\epsilon > 0$ such that
\begin{eqnarray*}
\mathfrak{m}_f(x) &\ge& c \, \mathrm{dist}\big(x, (\partial f)^{-1}(0)\big) \quad \textrm{ for all } \quad x \in \mathbb{B}_{\epsilon}(\bar{x})
\end{eqnarray*}
implies the existence of constants $c' > 0$ and $\epsilon' > 0$ satisfying
\begin{eqnarray*}
f(x) &\ge& f(\bar{x}) + c'\, \mathrm{dist}\big(x, (\partial f)^{-1}(0)\big)^{2} \quad \textrm{ for all } \quad x \in \mathbb{B}_{\epsilon'}(\bar{x}),
\end{eqnarray*}
where $\mathrm{dist}\big(x, (\partial f)^{-1}(0)\big)$ stands for the Euclidean distance from $x$ to $(\partial f)^{-1}(0).$ In \cite[Remark~3.4]{Drusvyatskiy2014},  Drusvyatskiy and Ioffe conjectured that the converse is also true for semi-algebraic functions.
The next example shows that this conjecture does not hold in general.
}\end{remark}

\begin{example}\label{Counterexample}{\rm
Let $f \colon \mathbb{R}^2 \rightarrow \mathbb{R}, ({x}, {y}) \mapsto f(x, y),$ be the continuous semi-algebraic function defined by $f(x, y) := |{x}^2 - {y}^4|.$ A direct calculation shows that
$$\partial f(x, y) =
\begin{cases}
\{(2{x}, - 4{y}^3)\} & \textrm{ if  } \quad {x}^2 - {y}^4 > 0, \\
\{(-2{x}, 4{y}^3)\} & \textrm{ if  } \quad {x}^2 - {y}^4 < 0, \\
\{(2(2t - 1){x}, - 4(2t - 1){y}^3) \ | \ t \in [0, 1] \} & \textrm{ otherwise.}
\end{cases}$$
In particular, we have
\begin{eqnarray*}
f^{-1}(0) &=& (\partial f)^{-1}(0) \ = \ \{({x}, {y}) \in \mathbb{R}^2 \ | \ {x}^2 - {y}^4 = 0\}.
\end{eqnarray*}
Let $P(x, y) := x^2 - y^4.$ According to Kuo's work \cite[Corollaries~1~and~2]{Kuo1974} (see also \cite{HaHV2018}), we find constants $c' > 0$ and $\epsilon' > 0$ such that
\begin{eqnarray*}
|P(x, y)| &\ge& c'\, \mathrm{dist} ((x, y), P^{-1}(0))^2  \quad \textrm{ for all } \quad \|(x, y)\| \le \epsilon'.
\end{eqnarray*}
Since $f$ is just the absolute of $P,$ it holds that
\begin{eqnarray*}
f(x, y) &\ge& c'\, \mathrm{dist} ((x, y), (\partial f)^{-1}(0))^2  \quad \textrm{ for all } \quad \|(x, y)\| \le \epsilon'.
\end{eqnarray*}

On the other hand, for all $t \in \mathbb{R}$ we have
\begin{eqnarray*}
\mathrm{dist}\big ((0, {{t}}), (\partial f)^{-1}(0) \big) 
& = & \mathrm{dist}\big ((0, {{t}}), \{({x}, {y}) \in \mathbb{R}^2 \ | \ {x}^2 - {y}^4 = 0\} \big) \\
& = & \mathrm{dist}\big ((0, {{t}}), \{({x}, {y}) \in \mathbb{R}^2 \ | \ {x} - {y}^2 = 0\} \big) \\
& = & \min\{\big (x^2 + (y - t)^2 \big)^{1/2} \ | \ {x} - {y}^2 = 0\}  \\
& = & \min\{\big (y^4 + (y - t)^2 \big)^{1/2} \ | \ y \in \mathbb{R} \}.
\end{eqnarray*}
Let $g(t, y) := y^4 + (y - t)^2.$ Then it is easy to see that for each $t \in \mathbb{R},$ the function $\mathbb{R} \to \mathbb{R}, y \mapsto g(t, y),$ is a convex polynomial, and so it has a unique minimizer, say, $y(t).$ Clearly, $y(0) = 0$ and $\frac{\partial g}{\partial y}(t, y(t)) = 0$ for all $t.$ Note that $\frac{\partial g}{\partial y}g(0, 0) = 0$ and
$\frac{\partial^2 g}{\partial y^2}g(0, 0) = 2 \ne 0.$ By the implicit function theorem, then $y = y(t)$ is an analytic function on an open interval containing $0 \in \mathbb{R},$ and so we can write\footnote{Using the software Maple, it is easy to see that $y(t) = t - 2t^3 + o(t^3).$}
\begin{eqnarray*}
y(t) &=& a_1 t + a_2 t^2 + o(t^2) \quad \textrm{as} \quad t \to 0,
\end{eqnarray*}
for some $a_1, a_2 \in \mathbb{R}.$ Since $\frac{\partial g}{\partial y}(t, y(t)) \equiv 0,$ it follows easily that $a_1 = 1$ and $a_2 = 0.$ Consequently,
\begin{eqnarray*}
\mathrm{dist}\big ((0, {{t}}), (\partial f)^{-1}(0) \big) 
& = & \sqrt{g(t, y(t))} \ = \ t^2 + o(t^2)  \quad \textrm{ as } \quad t \to 0.
\end{eqnarray*} 
Therefore, 
\begin{eqnarray*}
\lim_{{{t}} \to 0} \frac{\mathfrak{m}_f(0, {{t}})} {\mathrm{dist}\big ((0, {{t}}), (\partial f)^{-1}(0) \big)} 
&=& \lim_{{{t}} \to 0} \frac{4t^3} {t^2 + o(t^2)}  \ = \  0,
\end{eqnarray*}
which implies that there are no constants $c > 0$ and $\epsilon > 0$ such that
\begin{eqnarray*}
\mathfrak{m}_f(x, y) &\ge& c\, \mathrm{dist}\big((x, y), (\partial f)^{-1}(0)\big)  \quad \textrm{ for all } \quad \|(x, y)\| \le \epsilon.
\end{eqnarray*}
Consequently, there are no constants $c > 0$ and $\epsilon > 0$ such that
\begin{eqnarray*}
\mathfrak{m}_f(x, y) &\ge& c\, |f(x, y)|^{\frac{1}{2}}  \quad \textrm{ for all } \quad \|(x, y)\| \le \epsilon.
\end{eqnarray*}
}\end{example}

The next corollary determines constants, which correspond to sharp local minimizers.

\begin{corollary}
Under the assumptions of Theorem~\ref{Theorem42}, suppose that $\alpha \ge \alpha_*.$ Then for any constant $c  \in (0,  a_*)$ there exists $\epsilon > 0$ such that
\begin{eqnarray*}
f(x) &\ge& f(\bar{x}) + c \|x - \bar{x}\|^\alpha \quad \textrm{ for all } \quad x \in \mathbb{B}_{\epsilon}(\bar{x}),
\end{eqnarray*}
where $a_*  :=  \min\{a_k \ | \ k \in K \textrm{ and } \alpha_k = \alpha_*\}.$
\end{corollary}

\begin{proof}
This follows immediately from the argument given at the beginning of the proof of Theorem~\ref{Theorem42}.
\end{proof}

We finish this section with the following remark.

\begin{remark}{\rm
Let $\mathscr{L}_1$ and $\mathscr{L}_2$ be the smallest possible exponents $\alpha$ and $\theta,$ respectively, for which there exist positive constants $c$ and $\epsilon$ such that for all $x \in \mathbb{B}_\epsilon(\bar{x})$ the following inequalities hold:
\begin{eqnarray*}
|f(x) - f(\bar{x})|  &\ge&  c\, \mathrm{dist}\big(x, f^{-1}(0)\big)^{\alpha} \quad \textrm{ and } \quad  \mathfrak{m}_f(x) \ \ge \  c \, |f(x) - f(\bar{x})|^{\theta}.
\end{eqnarray*}
It is well-known (see, for example, \cite[Lemma~3.3]{HaHV2017}) that
\begin{eqnarray*}
\mathscr{L}_2 &\ge& 1 - \frac{1}{\mathscr{L}_1},
\end{eqnarray*}
and the inequality may be strict (for instance, we have $\mathscr{L}_1 = 2$ and $\mathscr{L}_2 > \frac{1}{2}$ for the function $f$ in Example~\ref{Counterexample}). On the other hand, if $\bar{x}$ is an isolated local minimizer of $f,$ then it follows from Theorem~\ref{Theorem42} that the (\L ojasiewicz) exponents $\mathscr{L}_1$ and $\mathscr{L}_2$ can be computed in terms of the tangency variety of $f$:
\begin{eqnarray*}
\mathscr{L}_1 &=& \alpha_* \quad \textrm{ and } \quad \mathscr{L}_2  \ = \ 1 - \frac{1}{\alpha_*}.
\end{eqnarray*}
Also note that there are formulas computing the exponents $\mathscr{L}_1$ and $\mathscr{L}_2$ when $f$ is an analytic function in two variables, see \cite{HaHV2018, Kuo1974, Nguyen2017}. So it would be interesting to compute these exponents in the general case. This question will be explored in our future research work.
}\end{remark}

\section{Examples} \label{Section5}

In this section we will provide an algorithmical method to find all the numbers $a_k$ and $\alpha_k$ of a given polynomial in two variables and to identify the kind of phenomena which occur for a given point: saddle point, local minimizer, isolated local minimizer. The method is as follows: Assume that $f$ is a polynomial function in two variables $(x, y) \in \mathbb{R}^2$ with coefficients in $\mathbb{Q}.$ For simplicity, we will assume that the point of interest is the origin $(0, 0) \in \mathbb{R}^2.$ By definition, then $\Gamma(f) = \{(x, y) \in \mathbb{R}^2 \ | \ g(x, y) = 0\},$ where $g \colon \mathbb{R}^2 \to \mathbb{R}$ is the polynomial function defined by
\begin{eqnarray*}
g(x, y) &:=& y \frac{\partial f}{\partial x} - x \frac{\partial f}{\partial y}.
\end{eqnarray*}
In particular, the tangency variety $\Gamma(f)$ is a curve; so are the components $\Gamma_1, \ldots, \Gamma_k.$

For each $k = 1, \ldots, p,$ let $\phi_k \colon (-\delta, \delta) \rightarrow \mathbb{R}^2$ be an analytic curve  such that $\phi_k(0) = 0$ and $\phi_k  \big ((0, \delta) \big) = \Gamma_k.$ We can write
\begin{eqnarray*}
\|\phi_k(t)\| &=& c_k t^{m_k} + o(t^{m_k}) \quad \textrm{ as } \quad t \to 0^+,
\end{eqnarray*}
where $c_k$ is a positive constant and $m_k$ is a positive integer. For $\delta > 0$ small enough, the function $(0, \delta) \to \mathbb{R}, t \mapsto \|\phi_k(t)\|,$ is strictly increasing, so it has an inverse function, say, $t = \psi_k(s).$ Then for all $s > 0$ small enough we have $\phi_k \circ \psi_k(s) \in \Gamma_k,$ $\|\phi_k \circ \psi_k(s)\| = s,$ and
\begin{eqnarray*}
\psi_k(s) &=& c_k^{-\frac{1}{m_k}} s^{\frac{1}{m_k}}  + o(s^{\frac{1}{m_k}}) \quad \textrm{ as } \quad  s \to 0^+.
\end{eqnarray*}

If $k \not \in K,$ then $f \circ \phi_k (t)  = f(0, 0)$ for all $t \in (0, \delta).$ Assume that $k \in K.$ We have
\begin{eqnarray*}
f \circ \phi_k  (t) &=& f(0, 0) + \tilde{a}_k t^{\tilde{\alpha}_k} + o(t^{\tilde{\alpha}_k})  \quad \textrm{ as } \quad t \to 0^+,
\end{eqnarray*}
where $\tilde{a}_k \in \mathbb{R}, \tilde{a}_k \ne 0,$ and $\tilde{\alpha}_k \in \mathbb{N}, \tilde{\alpha}_k > 0.$ By substituting $t = \psi_k(s)$ in the above expression, we get
\begin{eqnarray*}
f \circ \phi_k  \circ \psi_k (s) &=& f(0, 0) + \tilde{a}_k c_k^{-\frac{\tilde{\alpha}_k}{m_k}} 
s^{\frac{\tilde{\alpha}_k}{m_k}} + o(s^{\frac{\tilde{\alpha}_k}{m_k}})  \quad \textrm{ as } \quad s \to 0^+.
\end{eqnarray*}
Consequently, the following relations hold:
\begin{eqnarray}\label{Eq06}
a_k &=& \tilde{a}_k c_k^{-\frac{\tilde{\alpha}_k}{m_k}} \quad \textrm{ and } \quad \alpha_k \ = \ \frac{\tilde{\alpha}_k}{m_k}.
\end{eqnarray}

Now we perform the following steps:
\begin{itemize}
\item If the polynomial $g$ is not regular in $y,$ make a linear change of coordinates so that it becomes regular in $y.$\footnote{Write $g = g_m + g_{m + 1} + \cdots,$ where $g_m \not \equiv 0$ and each $g_k, k \ge m$ is a homogeneous polynomial of degree $k.$ Then $g$ is said to be {\em regular} in $y$ (of order $m$) if $g_m(0, 1) \ne 0.$ It is not hard to see that for almost all linear mappings $L$ from $\mathbb{R}^2$ into itself, the compose function $\mathbb{R}^2 \to \mathbb{R}, (x, y) \mapsto g \circ L(x, y),$ is regular in $y.$}

\item If $g$ is not square-free, factor $g = g_1^{n_1} \cdots g_l^{n_l},$ where all $g_i$ are square-free polynomials with coefficients in $\mathbb{Q}$ and all $n_i$ are positive integers.\footnote{A polynomial is called {\em square-free} if it does not have multiple factors.} This can be done by  greatest common divisor computations.

\item Compute the Puiseux expansions of the solutions for $y$ of the equation $g = 0$ (or $g_1 \cdots g_l = 0$ if $g$ is not square-free), as $x \to 0.$ This can be done using rational Puiseux expansions over $\mathbb{Q}$ (cf. \cite{Duval1989}), and we get solutions of the form $(x = c t^m; y = y(t)),$ where $c$ is a nonzero constant in $\mathbb{Q},$ $m$ is a positive integer, and $y(t)$ is a power series in $t$ with coefficients in a finite algebraic extension of $\mathbb{Q}.$ 

\item Construct the ordered lists of all components $\Gamma_k$ of $\Gamma(f).$ This can be done by finding the real branches, which means all Puiseux expansions with real coefficients. 

\item For each component $\Gamma_k,$ compute the numbers $\tilde{a}_k$ and $\tilde{\alpha}_k.$ This can be done by substituting the Puiseux expansions in the polynomial $f.$ 
\item Finally, the numbers ${a}_k$ and ${\alpha}_k$ are obtained by using \eqref{Eq06}.
\end{itemize}

Then we have all the information needed to apply Theorems \ref{Theorem41} and \ref{Theorem42}.

The computations can be performed with the software Maple, using the command ``puiseux" of the package ``algcurves" for the rational Puiseux expansions.

\begin{example}{\rm
Let $f(x, y) := 2\,{x}^{5}y+{x}^{4}-{y}^{3}+xy.$ By definition, $\Gamma(f) = g^{-1}(0),$ where
\begin{eqnarray*}
g(x, y) &:=& -2\,{x}^{6}+10\,{x}^{4}{y}^{2}+4\,{x}^{3}y+3\,x{y}^{2}-{x}^{2}+{y}^{2}.
\end{eqnarray*}
Since $g$ is regular in $y,$ we can compute the Puiseux expansions of the solutions of $g = 0$ and put them in a list. 
\begin{verbatim}
> PG := convert(puiseux(g, x = 0, y, 5, t), list);
\end{verbatim}

$\qquad \Big[[
x = t, y = -t + \frac{3}{2}\,{t}^{2} -{\frac {43}{8}}\,{t}^{3}+ {\frac {231}{16}}\,{t}^{4}], \qquad 
[x = t, y = t-\frac{3}{2}\,{t}^{2} +{\frac {11}{8}}\,{t}^{3}-{\frac {39}{16}}\,{t}^{4}] \Big].$

We next substitute these expansions in $f.$
\begin{verbatim}
> series(algsubs(x = t, algsubs(PG[1, 2], f)), t = 0, 5);
\end{verbatim}

$\qquad -{t}^{2}+{\frac {5}{2}}{t}^{3}-{\frac {71}{8}}{t}^{4}+O \left( {t}^{5} \right).$

\begin{verbatim}
> series(algsubs(x = t, algsubs(PG[2, 2], f)), t = 0, 5);
\end{verbatim}

$\qquad {t}^{2}-{\frac {5}{2}}{t}^{3}+{\frac {55}{8}}{t}^{4}+O \left( {t}^{5}\right).$

From these computations we can see that for sufficiently small $\epsilon > 0,$ the set $\Gamma(f) \cap \mathbb{B}_\epsilon \setminus \{(0, 0)\}$ has four connected components $\Gamma_{\pm 1}$ and $\Gamma_{\pm 2},$ which are given, respectively, by the following parametrizations:
\begin{eqnarray*}
\phi_{\pm 1}(t) &:=& \left(t, -t + \frac{3}{2}\,{t}^{2} -{\frac {43}{8}}\,{t}^{3}+ {\frac {231}{16}}\,{t}^{4} + o \left( {t}^{4} \right)\right),\\
\phi_{\pm 2}(t) &:=& \left(t, t-\frac{3}{2}\,{t}^{2} +{\frac {11}{8}}\,{t}^{3}-{\frac {39}{16}}\,{t}^{4} + o \left( {t}^{4} \right)\right),
\end{eqnarray*}
where $t \to 0^{\pm}.$ It is clear that $\|\phi_{\pm k}(t)\| = \sqrt{2}t + o(t)$ for $k = 1, 2,$ which yields $c_{\pm k} = \sqrt{2}$ and $m_{\pm k} = 1.$ Furthermore, we have 
\begin{eqnarray*}
f \circ \phi_{\pm 1}(t) & = & -{t}^{2}+{\frac {5}{2}}{t}^{3}-{\frac {71}{8}}{t}^{4}+O \left( {t}^{5} \right) , \\
f \circ \phi_{\pm 2}(t)  & = & {t}^{2}-{\frac {5}{2}}{t}^{3}+{\frac {55}{8}}{t}^{4}+O \left( {t}^{5}\right).
\end{eqnarray*}
It follows that $K = \{\pm 1, \pm 2\}$ and
\begin{eqnarray*}
\tilde{a}_{\pm 1} &=&  - 1\quad  \textrm{ and } \quad \tilde{a}_{\pm 2} \ = \  1,\\
\tilde{\alpha}_{\pm 1} &=& \ \ \,2 \quad  \textrm{ and } \quad \tilde{\alpha}_{\pm 2} \ = \  2.
\end{eqnarray*}
By \eqref{Eq06}, then
\begin{eqnarray*}
a_{\pm 1} &=& - \frac{1}{2} \quad  \textrm{ and } \quad a_{\pm 2} \ = \  \frac{1}{2},\\
\alpha_{\pm 1} &=& \quad 2 \quad  \textrm{ and } \quad \alpha_{\pm 2} \ = \  2.
\end{eqnarray*}
Since $a_{\pm 1} < 0 < a_{\pm 2},$ we deduce from Theorem~\ref{Theorem41} that the origin is a saddle point of $f.$ 
}\end{example}

\begin{example}{\rm
Let $f(x, y) := x^2(x^2y^2+1).$ We have $\Gamma(f) = g^{-1}(0),$ where
\begin{eqnarray*}
g(x, y) &:=& -2x^5y+4x^3y^3+2xy.
\end{eqnarray*}
Since $g$ is not regular in $y,$ we first perform the linear change of coordinates; for example, let $x := X + Y$ and $y := X - Y.$ Here is the Maple code:
\begin{verbatim}
> F := subs({x = X + Y, y = X - Y}, f); 
\end{verbatim}
$\qquad {X}^{6}+2\,Y{X}^{5}-{X}^{4}{Y}^{2}-4\,{Y}^{3}{X}^{3}-{X}^{2}{Y}^{4}+2
\,X{Y}^{5}+{Y}^{6}+{X}^{2}+2\,XY+{Y}^{2}.$
\begin{verbatim}
> G := subs({x = X + Y, y = X - Y}, g);
\end{verbatim}
$\qquad 2\,{X}^{6}-8\,Y{X}^{5}-22\,{X}^{4}{Y}^{2}+22\,{X}^{2}{Y}^{4}+8\,X{Y}^{
5}-2\,{Y}^{6}+2\,{X}^{2}-2\,{Y}^{2}.$

We next compute the Puiseux expansions of the solutions of $G = 0$ and put them in a list. 
\begin{verbatim}
> PG := convert(puiseux(G, X = 0, Y, 5, t), list);
\end{verbatim}

$\qquad \Big[[X = t, Y = t], \qquad [X = t, Y = -t],$

$\qquad  \ [X=t,Y=-12\,{t}^{4}{\it RootOf} \left( {{\it \_Z}}^{4}+1 \right) -8\,{
t}^{3} \left( {\it RootOf} \left( {{\it \_Z}}^{4}+1 \right)  \right) ^{2}$

$\hfill -4\,{t}^{2} \left( {\it RootOf} \left( {{\it \_Z}}^{4}+1 \right) 
 \right) ^{3}+t+{\it RootOf} \left( {{\it \_Z}}^{4}+1 \right)] \Big].$

Since the third expansion is not real (and is not zero at $t = 0),$ we only substitute the first two expansions in $F.$
\begin{verbatim}
> series(algsubs(X = t, algsubs(PG[1, 2], F)), t = 0, 5);
\end{verbatim}

$\qquad 4t^2.$

\begin{verbatim}
> series(algsubs(X = t, algsubs(PG[2, 2], F)), t = 0, 5);
\end{verbatim}

$\qquad 0.$

From these computations we can see that for sufficiently small $\epsilon > 0,$ the set $\Gamma(f) \cap \mathbb{B}_\epsilon \setminus \{(0, 0)\}$ has four connected components $\Gamma_{\pm 1}$ and $\Gamma_{\pm 2},$ which are given, respectively, by the following parametrizations:
\begin{eqnarray*}
\phi_{\pm 1}(t) &:=& \left(\pm 2t, 0 \right) \quad \textrm{ and} \quad \phi_{\pm 2}(t) \ := \ \left(0, \pm 2t \right).
\end{eqnarray*}
Clearly, $\|\phi_{\pm k}(t)\| = 2t$ for $k = 1, 2,$ and so $c_{\pm k} = {2}$ and $m_{\pm k} = 1.$ Furthermore, we have
\begin{eqnarray*}
f \circ \phi_{\pm 1}(t) &=& 4t^2 \quad \textrm{ and} \quad f \circ \phi_{\pm 2}(t) \ = \  0.
\end{eqnarray*}
It follows that $K = \{\pm 1\}$ and
\begin{eqnarray*}
\tilde{a}_{\pm 1} &=&  4 \quad  \textrm{ and } \quad \tilde{\alpha}_{\pm 1} \ = \ 2.
\end{eqnarray*}
From \eqref{Eq06} we obtain
\begin{eqnarray*}
a_{\pm 1} &=& 1 \quad  \textrm{ and } \quad a_{\pm 2} \ = \  0,\\
\alpha_{\pm 1} &=& 2.
\end{eqnarray*}
By Theorem~\ref{Theorem41}, the origin is a nonisolated local minimizer of $f.$ 
}\end{example}

\begin{example}{\rm
Let $f(x, y) := -{x}^{7}{y}^{5}+2\,{y}^{4}+{x}^{2}.$ We have $\Gamma(f) = g^{-1}(0),$ where
\begin{eqnarray*}
g(x, y) &:=& 5\,{x}^{8}{y}^{4}-7\,{x}^{6}{y}^{6}-8\,x{y}^{3}+2\,xy.
\end{eqnarray*}
Then by similar computations as in the above example, it is easy to see that for sufficiently small $\epsilon > 0,$ the set $\Gamma(f) \cap \mathbb{B}_\epsilon \setminus \{(0, 0)\}$ has four connected components $\Gamma_{\pm 1}$ and $\Gamma_{\pm 2},$ which are given, respectively, by the following parametrizations:
\begin{eqnarray*}
\phi_{\pm 1}(t) &:=& \left(\pm 2t, 0 \right) \quad \textrm{ and} \quad  \phi_{\pm 2}(t) \ := \ \left(0, \pm 2t \right).
\end{eqnarray*}
It is clear that $\|\phi_{\pm k}(t)\| = 2t$ for $k = 1, 2,$ and so $c_{\pm k} = {2}$ and $m_{\pm k} = 1.$ Furthermore, we have
\begin{eqnarray*}
f \circ \phi_{\pm 1}(t) &=& 4t^2 \quad \textrm{ and} \quad  f \circ \phi_{\pm 2}(t) \ = \ 32t^4.
\end{eqnarray*}
It follows that $K = \{\pm 1, \pm 2\}$ and
\begin{eqnarray*}
\tilde{a}_{\pm 1} &=&  4 \quad  \textrm{ and } \quad \tilde{a}_{\pm 2} \ = \  32,\\
\tilde{\alpha}_{\pm 1} &=& 2 \quad  \textrm{ and } \quad \tilde{\alpha}_{\pm 2} \ = \  4.
\end{eqnarray*}
From \eqref{Eq06} we obtain
\begin{eqnarray*}
a_{\pm 1} &=& 1 \quad  \textrm{ and } \quad a_{\pm 2} \ = \  2,\\
\alpha_{\pm 1} &=& 2 \quad  \textrm{ and } \quad \alpha_{\pm 2} \ = \  4.
\end{eqnarray*}
By Theorems~\ref{Theorem41} and \ref{Theorem42}, the origin is an $\alpha$th order sharp local minimizer of $f$ for all $\alpha \ge \alpha_* = \max_{k = \pm 1, \pm 2} \alpha_k = 4.$
}\end{example}

\section{Conclusions}
This paper considers local minimizers of semi-algebraic functions. In terms of the tangency variety, we have presented necessary and sufficient conditions for optimality. We have also shown relationships between generalized notions of sharp minima, strong metric subregularity and the \L ojasiewicz gradient inequality; these relations may easily fail when the minimizer in question is not isolated. The constrained case will be studied in future research.

\subsection*{Acknowledgments}
The author wishes to thank the anonymous referees and the Associate Editor for their valuable comments that led to an improved version of the original submission.

\bibliographystyle{siam}

\begin{thebibliography}{99}

\bibitem{Absil2005}
{\sc P.~A. Absil, R.~Mahony, and B.~Andrews}, {\em Convergence of the iterates
  of descent methods for analytic cost functions}, SIAM J. Optim., 16 (2005),
  pp.~531--547.

\bibitem{AragonArtacho2008}
{\sc F.~J. Arag\'on-Artacho and M.~H. Geoffroy}, {\em Characterization of
  metric regularity of subdifferentials}, J. Convex Anal., 15 (2008),
  pp.~365--380.

\bibitem{AragonArtacho2014}
\leavevmode\vrule height 2pt depth -1.6pt width 23pt, {\em Metric subregularity
  of the convex subdifferential in {{B}}anach spaces}, J. Nonlinear Convex
  Anal., 15 (2014), pp.~35--47.

\bibitem{Attouch2010}
{\sc H.~Attouch, J.~Bolte, P.~Redont, and A.~Soubeyran}, {\em Proximal
  alternating minimization and projection methods for nonconvex problems: an
  approach based on the {{K}}urdyka--{{\L}}ojasiewicz inequality}, Math. Oper.
  Res., 35 (2010), pp.~438--457.

\bibitem{Attouch2013}
{\sc H.~Attouch, J.~Bolte, and B.~F. Svaiter}, {\em Convergence of descent
  methods for semi-algebraic and tame problems: proximal algorithms,
  forward-backward splitting, and regularized {{G}}auss--{{S}}eidel methods},
  Math. Program. Ser. A, 137 (2013), pp.~91--129.

\bibitem{BaroneNetto1984}
{\sc A.~Barone-Netto}, {\em Jet-detectable extrema}, Proc. Amer. Math. Soc., 92
  (1984), pp.~604--608.

\bibitem{Bertsekas1999}
{\sc D.~Bertsekas}, {\em Nonlinear programming}, Athena Scientific Optimization
  and Computation Series, Athena Scientific Publishers, Belmont, MA, third~ed.,
  2016.

\bibitem{Bochnak1998}
{\sc J.~Bochnak, M.~Coste, and M.-F. Roy}, {\em Real algebraic geometry},
  vol.~36, Springer, Berlin, 1998.

\bibitem{Bochnak1971}
{\sc J.~Bochnak and S.~{{\L}}ojasiewicz}, {\em A converse of the
  {{K}}uiper--{{K}}uo theorem}, in Proceedings of Liverpool Singularities
  Symposium, I (1969/70), vol.~192 of Lecture Notes in Math., Springer, Berlin,
  1971, pp.~254--261.

\bibitem{Bolte2006}
{\sc J.~Bolte, A.~Daniilidis, and A.~S. Lewis}, {\em The {{\L}}ojasiewicz
  inequality for nonsmooth subanalytic functions with applications to
  subgradient dynamical systems}, SIAM J. Optim., 17 (2006), pp.~1205--1223.

\bibitem{Bolte2007-2}
{\sc J.~Bolte, A.~Daniilidis, A.~S. Lewis, and M.~Shiota}, {\em Clarke
  subgradients of stratifiable functions}, SIAM J. Optim., 18 (2007),
  pp.~556--572.

\bibitem{Bolte2017}
{\sc J.~Bolte, T.~P. Nguyen, J.~Peypouquet, and B.~W. Suter}, {\em From error
  bounds to the complexity of first-order descent methods for convex
  functions}, Math. Program. Ser. A, 165 (2017), pp.~471--507.

\bibitem{Bolte2014}
{\sc J.~Bolte, S.~Sabach, and M.~Teboulle}, {\em Proximal alternating
  linearized minimization for nonconvex and nonsmooth problems}, Math. Program.
  Ser. A, 146 (2014), pp.~459--494.

\bibitem{Bonnas2000}
{\sc J.~F. Bonnans and A.~Shapiro}, {\em Perturbation analysis of optimization
  problems}, Springer, New York, 2000.

\bibitem{Clarke1990}
{\sc F.~H. Clarke}, {\em Optimization and nonsmooth analysis}, Classics in
  Applied Mathematics, SIAM, Philadelphia, PA, 1990.

\bibitem{Dontchev2009}
{\sc A.~L. Dontchev and R.~T. Rockafellar}, {\em Implicit functions and
  solution mappings. {{A}} view from variational analysis}, Springer Monographs
  in Mathematics, Springer, Dordrecht, 2009.

\bibitem{Drusvyatskiy2015}
{\sc D.~Drusvyatskiy and A.~D. Ioffe}, {\em Quadratic growth and critical point
  stability of semi-algebraic functions}, Math. Program. Ser. A, 153 (2015),
  pp.~635--653.

\bibitem{Drusvyatskiy2014}
{\sc D.~Drusvyatskiy, B.~S. Mordukhovich, and T.~T.~A. Nghia}, {\em
  Second-order growth, tilt-stability, and metric regularity of the
  subdifferential}, J. Convex Anal., 21 (2014), pp.~1165--1192.

\bibitem{Duval1989}
{\sc D.~Duval}, {\em Rational puiseux expansions}, Compositio Math., 70 (1989),
  pp.~119--154.

\bibitem{Frankel2015}
{\sc P.~Frankel, G.~Garrigos, and J.~Peypouquet}, {\em Splitting methods with
  variable metric for {{K}}urdyka--{{\L}}ojasiewicz functions and general
  convergence rates}, J. Optim. Theory Appl., 165 (2015), pp.~874--900.

\bibitem{Guo2018}
{\sc F.~Guo and T.~S. Ph\d{a}m}, {\em On types of degenerate critical points of
  real polynomial functions}, J. Symb. Comput., 99 (2020), pp. 108--126.
 
\bibitem{Gwozdziewicz1999}
{\sc J.~Gwo\'zdziewicz}, {\em The {{\L}}ojasiewicz exponent of an analytic
  function at an isolated zero}, Comment. Math. Helv., 74 (1999), pp.~364--375.

\bibitem{HaHV2018}
{\sc H.~V. H\`a}, {\em Computation of the {{\L}}ojasiewicz exponent for a germ
  of a smooth function in two variables}, Studia Math., 240 (2018),
  pp.~161--176.

\bibitem{HaHV2017}
{\sc H.~V. H\`a and T.~S. Ph\d{a}m}, {\em Genericity in polynomial
  optimization}, vol.~3 of Series on Optimization and Its Applications, World
  Scientific, Singapore, 2017.

\bibitem{Hardt1980}
{\sc R.~M. Hardt}, {\em Semi-algebraic local-triviality in semi-algebraic
  mappings}, Amer. J. Math., 102 (1980), pp.~291--302.

\bibitem{Ioffe2016-1}
{\sc A.~D. Ioffe}, {\em Metric regularity--a survey {{P}}art {{I.}}
  {{T}}heory}, J. Aust. Math. Soc., 101 (2016), pp.~188--243.

\bibitem{Ioffe2016-2}
\leavevmode\vrule height 2pt depth -1.6pt width 23pt, {\em Metric regularity--a
  survey {{P}}art {{II.}} {{A}}pplications}, J. Aust. Math. Soc., 101 (2016),
  pp.~376--417.

\bibitem{Klatte2002}
{\sc D.~Klatte and B.~Kummer}, {\em Nonsmooth equations in optimization.
  Regularity, calculus, methods and applications}, vol.~60 of Nonconvex
  Optimization and its Applications, Kluwer Academic Publishers, Dordrecht,
  2002.

\bibitem{Kuo1974}
{\sc T.~C. Kuo}, {\em Computation of {{\L}}ojasiewicz exponent of $f(x, y)$},
  Comment. Math. Helv., 49 (1974), pp.~201--213.

\bibitem{Lasserre2015}
{\sc J.~B. Lasserre}, {\em An introduction to polynomial and semi-algebraic
  optimization}, Cambridge University Press, Cambridge, 2015.

\bibitem{Li2015}
{\sc G.~Li, B.~Mordukhovich, and T.~S. Ph\d{a}m}, {\em New fractional error
  bounds for polynomial systems with applications to {{H}}\"olderian stability
  in optimization and spectral theory of tensors}, Math. Program. Ser. A, 153
  (2015), pp.~333--362.

\bibitem{Li2016}
{\sc G.~Li and T.~K. Pong}, {\em Douglas--{{R}}achford splitting for nonconvex
  optimization with application to nonconvex feasibility problems}, Math.
  Program. Ser. A, 159 (2016), pp.~371--401.

\bibitem{Li2018-2}
\leavevmode\vrule height 2pt depth -1.6pt width 23pt, {\em Calculus of the
  exponent of {{K}}urdyka--{{\L}}ojasiewicz inequality and its applications to
  linear convergence of first-order methods}, Found. Comput. Math., 18 (2018),
  pp.~1199--1232.

\bibitem{Mordukhovich2006}
{\sc B.~S. Mordukhovich}, {\em Variational analysis and generalized
  differentiation, I: Basic theory; II: Applications}, Springer, Berlin, 2006.

\bibitem{Mordukhovich2018}
\leavevmode\vrule height 2pt depth -1.6pt width 23pt, {\em Variational Analysis
  and Applications}, Springer, New York, 2018.

\bibitem{Mordukhovich2015}
{\sc B.~S. Mordukhovich and W.~Ouyang}, {\em Higher-order metric subregularity
  and its applications}, J. Global Optim., 63 (2015), pp.~777--795.

\bibitem{Nguyen2017}
{\sc H.~D. Nguyen, T.~S. Ph\d{a}m, and P.~D. Hoang}, {\em Topological
  invariants of plane curve singularities: {{P}}olar quotients and
  {{\L}}ojasiewicz gradient exponents}, Internat. J. Math., 30 (2019), 1950073, 19 pp.

\bibitem{Nocedal2006}
{\sc J.~Nocedal and S.~J. Wright}, {\em Numerical optimization}, Operations
  Research and Financial Engineering, Springer, New York, 2nd~ed., 2006.

\bibitem{PHAMTS2011}
{\sc T.~S. Ph\d{a}m}, {\em The {{\L}}ojasiewicz exponent of a continuous
  subanalytic function at an isolated zero}, Proc. Amer. Math. Soc., 139
  (2011), pp.~1--9.

\bibitem{Rockafellar1998}
{\sc R.~T. Rockafellar and R.~Wets}, {\em Variational analysis}, vol.~317 of
  Grundlehren Math. Wiss., Springer-Verlag, Berlin, 1998.

\bibitem{Ruszczynski2006}
{\sc A.~Ruszczy\'nski}, {\em Nonlinear optimization}, Princeton University
  Press, Princeton, NJ, 2006.

\bibitem{Dries1996}
{\sc L.~van~den Dries and C.~Miller}, {\em Geometric categories and o-minimal
  structures}, Duke Math. J., 84 (1996), pp.~497--540.

\bibitem{Zheng2015}
{\sc X.~Y. Zheng and K.~F. Ng}, {\em H\"older stable minimizers, tilt
  stability, and {{H}}\"older metric regularity of subdifferentials}, SIAM J.
  Optim., 25 (2015), pp.~416--438.
\end{thebibliography}

\end{document}